\documentclass[11pt]{article}

\usepackage[utf8]{inputenc}
\usepackage{amscd,amssymb}
\usepackage{amsthm,amsmath}
\usepackage{graphicx}
\usepackage{enumerate,amssymb,setspace}
\usepackage{amsfonts,amsthm,amscd}
\usepackage{mathtools}
\usepackage{cite}
\usepackage{mathtools}
\usepackage[all]{xy}
\usepackage{hyperref}
 \usepackage{comment}

\usepackage{caption}
\usepackage{tikz}
\usetikzlibrary{snakes}

\DeclareCaptionLabelFormat{opening}{}
\newtheorem{theorem}{Theorem}[section]
\newtheorem{lemma}[theorem]{Lemma}
\newtheorem{prop}[theorem]{Proposition}
\newtheorem{corollary}[theorem]{Corollary}
\newtheorem{conjecture}[theorem]{{Conjecture}}
\newtheorem{example}[theorem]{{Example}}
\newtheorem{problem}[theorem]{{Problem}}
\newtheorem{question}[theorem]{{Question}}
\newtheorem{definition}[theorem]{{Definition}}
\newtheorem{claim}[theorem]{{Claim}}

\theoremstyle{remark}
\newtheorem{remark}[theorem]{Remark}

\makeatletter\@addtoreset{equation}{section} \makeatother

 \def\be{\beta}
\def\o{\omega} 
 
 \def\eps{\epsilon}

\newcommand{\dbar}{\bar\partial}
\newcommand{\ddbar}{\partial\dbar}

\def\K{K\"ahler } 
\def\KE{K\"ahler--Einstein }

\def\h#1{\hbox{#1}}

\def\strutdepth{\dp\strutbox}
\def\specialstar{\vtop to \strutdepth{
    \baselineskip\strutdepth
    \vss\llap{$\star$\ \ \ \ \ \ \ \ \  }\null}}
\def\marginalstar{\strut\vadjust{\kern-\strutdepth\specialstar}}

\def\marginal#1{\strut\vadjust{\kern-\strutdepth
    {\vtop to \strutdepth{
    \baselineskip\strutdepth
    \vss\llap{{ \small #1 }}\null}
    }}
    }

\def\text{\textstyle}

\def\q{\quad} \def\qq{\qquad}

\def\ra{\rightarrow}

\def\i{\sqrt{-1}}

\def\Fn{\mathbb{F}_n}

\def\exc{\operatorname{exc}}

\newcommand{\PP}{{\mathbb P}} \newcommand{\RR}{\mathbb{R}}
\newcommand{\QQ}{\mathbb{Q}} 
\newcommand{\ZZ}{{\mathbb Z}} \newcommand{\NN}{{\mathbb N}}
\newcommand{\FF}{{\mathbb F}}

\def\beq{\begin{equation}}
\def\eeq{\end{equation}}
\def\bpf{\begin{proof}}
\def\epf{\end{proof}}

\def\aasc{\h{$\h{\rm AA}(S,C)$}}

\def\aaxld{\h{$\h{\rm AA}(X,L,D)$}}

\def\bclaim{\begin{claim}}
\def\eclaim{\end{claim}}
\def\bdefin{\begin{definition}}
\def\edefin{\end{definition}}
\def\bcor{\begin{corollary}}
\def\ecor{\end{corollary}}
\def\bthm{\begin{theorem}}
\def\ethm{\end{theorem}}
\def\bconj{\begin{conjecture}}
\def\econj{\end{conjecture}}
\def\blem{\begin{lemma}}
\def\elem{\end{lemma}}
\def\blemma{\begin{lemma}}
\def\elemma{\end{lemma}}
\def\bprop{\begin{prop}}
\def\eprop{\end{prop}}
\def\bremark{\begin{remark}}
\def\eremark{\end{remark}}

\def\eaeq{\end{aligned}}
\def\baeq{\begin{aligned}}

\def\lb{\label}

\def\ga{\gamma}
\def\de{\delta}

\newcommand\blfootnote[1]{%
	\begingroup
	\renewcommand\thefootnote{}\footnote{#1}%
	\addtocounter{footnote}{-1}%
	\endgroup
}

\def\wt{\widetilde}

\def\thh#1{${\hbox{#1}}^{{\hbox{\notsosmall th}}}$}
\font\notsosmall=cmr7

\title{High-dimensional convex sets arising in algebraic geometry}
\author{Yanir A. Rubinstein}
\begin{document}

	\bibliographystyle{amsalpha}
	\maketitle

\font\itnotsosml=cmti7

\def\thh#1{${\hbox{#1}}^{{\hbox{\notsosmall th}}}$}
\def\thhnotsosmall#1{${\hbox{#1}}^{\hbox{\small th}}$}
\def\ndnotsosmall#1{${\hbox{#1}}^{\hbox{\small nd}}$}

\def\thhit#1{\hbox{${\hbox{#1}}^{\,{\hbox{\itnotsosml th}}}$}}
\def\thhnotsosmallit#1{${\hbox{#1}}^{\hbox{\itsml th}}$}
\def\ndnotsosmallit#1{${\hbox{#1}}^{\hbox{\itsml nd}}$}

\vglue-0.3cm
\centerline{\it Dedicated to Bo Berndtsson  on the occassion of his
6\thhit{8} birthday }

	\begin{abstract}

We introduce an asymptotic notion of positivity
in algebraic geometry that turns out to be related
to some high-dimensional convex
sets.
The dimension of the convex sets grows with the number of birational
operations. In the case of complex
surfaces we explain how to associate a linear
program to certain sequences of blow-ups and how to reduce
verifying the asymptotic log positivity to checking feasibility
of the program.
	\end{abstract}
	

	\blfootnote{Research supported by
		NSF grant DMS-1515703. I am grateful to
I. Cheltsov and J. Martinez-Garcia for collaboration over
the years on these topics, to 
G. Livshyts for the invitation to speak in the High-dimensional Seminar
in Georgia Tech,
to the editors for the invitation to contribute
to this volume, and to a referee for a careful reading and catching many typos.
	}


    
	\section{Introduction}

Convex sets have long been known to appear
in algberaic geometry. A well-known example whose
origins can be traced to Newton and Minding are 
the convex polytopes associated to toric varieties
\cite{Fulton,Gromov,Oda}, also known as Delzant polytopes in the
symplectic geometry literature \cite{Delzant}. 
In recent years, this notion has been further extended
to any projective variety, the so-called
Newton--Okounkov bodies (or `nobodies'). In the most
basic level, avoiding a formal definition, such a
body is a compact convex body (not necessarily a polytope) in $\RR^n$ 
associated to two pieces of data: a nested sequence of subvarieties inside
a projective variety of complex dimension $n$, and a line bundle over
the variety.
Among other things, beautiful relations between the notion of {\it volume}
in algebraic geometry and the volume of these bodies
have been proved  \cite{KK,LM}. 

The purpose of this note, motivated
by a talk in the High-dimensional Seminar at Georgia Tech
in Decemeber 2018, is to associate another type of convex 
bodies to projective varieties. The main novelty is that
this time the convex bodies can have {\it unbounded dimension }
while the projective variety has {\it fixed dimension} (which, for most of the discussion,
will be in fact 2 (i.e., real dimension 4)). 
In fact, the asymptotic behavior of the bodies as the dimension
grows (on the convex side) corresponds to increasingly complicated 
birational operations such as blow-ups (on the algebraic side).
Rather than volume, we will be interested
in intersection properties of these bodies.
This gives the first relation between
algebraic geometry and asymptotic convex geometry that we are
aware of.

This note will be aimed at geometers on both sides
of the story---convex and algebraic. Therefore, it
will aim to recall at least some elementary notions on both sides.
Clearly, a rather unsatisfactory compromise had to be
made on how much background to provide, but it is our
hope that at least the gist of the ideas are conveyed
to experts on both sides of the story.

\subsection{Organization}

We start by introducing asymptotic log positity in \S\ref{alpSec}.
It is a generalization of the notion of positivity of divisors
in algebraic geometry, and the new idea is that it 
concerns pairs of divisors in a particular way. 
In \S\ref{ample-angles} we associate with this new
notion of positivity a convex body, the body of ample angles.
In \S\ref{ALFSec} we explain how two previously defined
classes of varieties 
(asymptotically log Fano varieties and asymptotically log 
canonically polarized varieties) fit in with this picture.
The problem of classifying two-dimensional 
asymptotically log Fano varieties has been posed in 2013
by Cheltsov and the author and is recalled 
(Problem \ref{relationSec}) as well as the progress on it
so far. In \S\ref{convexSec} we make further progress on
this problem by making a seemingly new connection between 
birational geometry and linear programming, in the process
explaining how birational blow-up operations yield convex
bodies of increasingly high dimension.
Our main results, Theorems \ref{TailThm0} and \ref{TailThm1}, 
first reduce the characterization of
``tail blow-ups" (Definition \ref{TailDef2}) that
preserve the asymptotic log Fano property to checking
the feasibility of a certain linear program and, second, show that
the linear program can be simplified. The proof,
which is the heart of this note,
involves associating a linear program to the sequence
of blow-ups and characterizing when it is feasible.
The canonically polarized case will be discussed elsewhere.
A much more extensive classification of asymptotically 
log del Pezzo surfaces is the topic of a 
forthcoming work \cite{MGR} and we refer the reader to
Remark \ref{tailremark-3} for the relation between Theorems 
 \ref{TailThm0} and \ref{TailThm1} and 
that work.

This note is dedicated to Bo Berndtsson, whose
contributions to the modern understanding
and applications of convexity and positivity on the one hand,
and whose generosity, passion, curiosity, and wisdom on the other hand,
have had a lasting and profound influence on the author over the years.

\section{Asymptotic log positivity}
\lb{alpSec}

The key new algebraic notion that gives birth to
the convex bodies alluded to above is 
{\it asymptotic log positivity}. 
Before introducing
this notion let us first pause to explain
the classical notion of positivity, absolutely
central to algebraic geometry, on 
which entire books have been written 
\cite{Lazarsfeld-positivity}.

\subsection{Positivity}

Consider a projective manifold $X$, i.e., a smooth complex
manifold that can be embedded in  
some complex projective space $\PP^N$.
In algebraic geometry, one is often interested in 
notions of positivity. Incidentally, these notions
are complex generalizations/analogues of notions of
convexity. In discussing these notions one
interchangeably switches
between line bundles, divisors, and cohomology classes%
\footnote{A great place to read about this trinity is the cult classic text
of Griffiths--Harris \cite[\S1.1]{GH} that was
written when the latter was a graduate 
student of the former.}. 
Complex codimension 1 submanifolds
of $X$ are locally defined by a single equation.
Formal sums (with coefficients in $\ZZ$) of such submanifolds
is a  {\it divisor } 
(when the formal sums are taken with coefficients in $\QQ$ or $\RR$ this is
called a $\QQ$-divisor or a $\RR$-divisor).
By the Poincar\'e duality between homology and cohomology,
a (homology class of a) divisor $D$ gives rise to a cohomology
class $[D]$ in $H^2(M,\FF)$ with $\FF\in\{\ZZ,\QQ,\RR\}$. On the
other hand a line bundle is, roughly, a way to patch up
local holomorphic functions on $X$ to a global object (a
`holomorphic section' of the bundle). The zero locus of
such a section 
is then a formal sum of complex hypersurfaces, a divisor. 
E.g., the holomorphic sections of the hyperplane bundle in $\PP^N$ 
are linear equations in the projective coordinates 
$[z_0:\ldots:z_N]$, whose associated divisors are the
hyperplanes $\PP^{N-1}\subset \PP^N$. The associated 
cohomology class, denoted $[H]$, is the generator of $H^2(\PP^N,\ZZ)\cong\ZZ$.
The anticanonical bundle of $\PP^N$, on the other hand, is 
represented by $[(N+1)H]$ and its holomorphic sections are
homogeneous polynomials of degree $N+1$ in $z_0,\ldots,z_N$.
Either way, both of these bundles are prototypes of positive
ones, a notion we turn to describe.

Now perhaps the simplest way to define positivity, at least for a differential
geometer, is to consider the cohomology class part of the story.
A class $\Omega$
in $H^2(X,\ZZ)$ admits a representative $\o$ (written $\Omega=[\omega]$), 
a real 2-form, that can
be written locally as $\sqrt{-1}\sum_{i,j=1}^ng_{i\bar j}dz^i\wedge\overline{dz^j}$
with $[g_{i\bar j}]$ a positive Hermitian matrix, and $z_1,\ldots,z_n$
are local holomorphic coordinates on $X$.  
Since a cohomology class can be associated to both line bundles
and divisors, this gives a definition of positivity for all three.
As a matter of terminology one usually speaks of a divisor being
`ample', while a cohomology class is referred to as `positive'.
For line bundles one may use either word.
A line bundle is called negative (the divisor `anti-ample') 
if its dual is positive.

The beauty of positivity is that it can be defined in many equivalent
ways. Starting instead with the line bundle $L$, we say
$L$ is positive if it admits a smooth Hermitian
metric $h$ with positive curvature 2-form $-\i\ddbar\log h=:c_1(L,h)$. 
By Chern--Weil theory the cohomology class $c_1(L)=[c_1(L,h)]$
is independent of $h$.

\subsection{Asymptotic log positivity}

We define asymptotic log positivity/negativity similarly, but now we
will consider pairs $(L,D)$ and allow for
asymptotic corrections along a divisor $D$ (in algebraic geometry
the word log usually refers to considering the extra
data of a divisor).
Let $D=D_1+\ldots + D_r$ be a divisor on $X$.
We say that 
$(L,D=D_1+\ldots + D_r)$ is {\it asymptotically log positive/negative} if 
$L-\sum_{i=1}^r(1-\be_i)D_i$
is positive/negative for all $\be=(\be_1,\ldots,\be_r)\in U\subset (0,1)^r$
with $0\in \overline U$. 
For the record, let us give a precise definition as 
well as two slight variants.

\begin{definition}
\label{ALPDef}
Let $L$ be a line bundle over a normal projective variety $X$,
and let $D=D_1+\ldots+D_r$ be a divisor, where
$D_i, i=1,\ldots,r$ are distinct $\QQ$-Cartier prime Weil divisors
on $X$. 

$\bullet\;$
We call $(L,D)$ asymptotically log positive/negative if 
$c_1(L)-\sum_{i=1}^r(1-\be_i)[D_i]$
is positive/negative for all $\be=(\be_1,\ldots,\be_r)\in U\subset (0,1)^r$
with $0\in \overline U$.

$\bullet\;$
We say $(L,D)$ is strongly asymptotically log positive/negative if 
$c_1(L)-\sum_{i=1}^r(1-\be_i)[D_i]$
is positive/negative for all $\be=(\be_1,\ldots,\be_r)\in (0,\eps)^r$
for some $\eps>0$.

$\bullet\;$
We say $(L,D)$ is log positive/negative if 
$c_1(L)-[D]$
is positive/negative.
\end{definition}

Note that log positivity implies strong 
asymptotic log positivity which implies
asymptotic log positivity (ALP). None of the reverse
implications hold, in general.

The usual notion of positivity can be recovered 
(by openness of the positivity property) if
one required the $\be_i$ to be close to 1.
By requiring the $\be_i$ to hover instead near $0$ 
we obtain a notion that is rather different, but
more flexible and still recovers positivity. Indeed, asymptotic log positivity
generalizes positivity, as $L$ is positive
if and only if $(L,D_1)$ is asymptotically log positive
where $D_1$ is a divisor associated to $L$.
However, the ALP property allows us to `break' $L$ into
pieces and then put different weights along them, so
that $(L,D)$ could be ALP even if $L$ itself is not positive.
Let us give a simple example. 

\begin{example} {\rm
\label{Ex1}
Let $X$ be  
the blow-up of $\PP^2$ at a point $p\in\PP^2$.
Let $f$ be a hyperplane containing $p$ and let $\pi^{-1}(f)$
denote the total transform (i.e., the pull-back), the union of two curves: the
exceptional curve $Z_1\subset X$ and another curve $F\subset X$
(such that $\pi(Z_1)=p, \pi(F)=f$).
Downstairs $f$ is ample, but $\pi^{-1}(f)$ 
fails to be positive along the exceptional
curve $Z_1$. However, $(\pi^{-1}(f),Z_1)$
is ALP. 

} \end{example}

This example is not quite illustrative, though,
since it is really encoded in a classical object
in algebraic called the Seshadri constant.
In fact in the example above one does not need
to take $\beta$ small, rather it is really
$1-\be$ that is the `small parameter'
(and, actually, any $\beta\in(0,1)$ works,
reflecting that the Seshadri constant is 1 here). 

A better example is as follows.

\begin{example} {\rm
\label{Ex2}
Let $X=\Fn$ be the $n$-th Hirzebruch surface, $n\in\NN$.
Let $-K_X$ be the anticanonical bundle. It is
positive if and only if $n=0,1$. In general,
$-K_X$ is linearly equivalent to the divisor
$2Z_n+(n+2)F$ where $Z_n$ is the unique $-n$-curve
on $X$ (i.e., $Z_n^2=-n$) and $F$ is a fiber (i.e., $F^2=0$). 
A divisor of the form $aZ_n+bF$ is ample
if and only if $b>na$. Thus $(-K_X,Z_n)$ is ALP
precisely for $\be\in(0,\frac 2n)$. 

} \end{example}

\section{The body of ample angles}
\label{ample-angles}

The one-dimensional
convex body $(0,\frac 2n)$ of Example \ref{Ex2} is the simplest 
that occurs in our theory. 
Let us define the bodies that are the topic of
the present note.

Let $D=\sum_{i=1}^rD_i$, and denote
\beq
\lb{Lbeta}
L_{\beta,D}
:=
L-\sum_{i=1}^r(1-\beta_i)D_i.
\eeq
The problem of determining whether a given 
pair $(L,D=\sum_{i=1}^rD_i)$ is ALP amounts to determining whether the set 
\begin{equation}
\begin{aligned}
\label{AAEq}
\h{\rm AA}_\pm(X,L,D):=\{\be=(\be_1,\ldots,\be_r)\in(0,1)^r\,:\, 
\h{$\pm L_{\beta,D}$ is ample} \}
\end{aligned}
\end{equation}
satisfies
$$
0\in\overline{\h{\rm AA}_\pm(X,L,D)}.
$$ 
Thus, this set is a fundamental object in the
study of asymptotic log positivity.

\begin{definition}
\label{AAEq}
We call $\h{\rm AA}_+(X,L,D)$ the body of ample angles of $(X,L,D)$,
and $\h{\rm AA}_-(X,L,D)$ the body of anti-ample angles of $(X,L,D)$. 
\end{definition}

\begin{remark}
{\rm 
The body of ample angles encodes both asymptotic log positivity
and the classical notion of nefness. Indeed, if 
$(1,\ldots,1)\in\overline{\h{\rm AA}_\pm(X,L,D)}$ 
then $\pm L$ is numerically effective (nef).
}
\end{remark}

\begin{lemma}
\label{}
When nonempty, $\h{\rm AA}_\pm(X,L,D)$ is an open convex body in $\RR^r$.
\end{lemma}

\begin{proof}
Suppose $\h{\rm AA}_+(X,L,D)$ is nonempty.
Openness is clear since positivity (and, hence,  ampleness) 
is an open condition on $H^2(X,\RR)$. For convexity,
suppose that $\be,\ga\in\h{\rm AA}_+(X,L,D)\subset\RR^r$.
Then, for any $t\in(0,1)$,
$$
\baeq
L_{t\be+(1-t)\ga,D}
&=
L-\sum_{i=1}^r(1-t\be_i-(1-t)\ga_i)D_i
\cr
&=
(t+1-t)L-\sum_{i=1}^r(t+1-t-t\be_i-(1-t)\ga_i)D_i
\cr
&=
t\big[L-\sum_{i=1}^r(1-\be_i)D_i\big]
+
(1-t)\big[L-\sum_{i=1}^r(1-\ga_i)D_i\big]
\eaeq
$$
is positive since the positive cone within
$H^2(X,\RR)$ is convex.
If $\be,\ga\in\h{\rm AA}_-(X,L,D)\subset\RR^r$ we get
$$
-L_{t\be+(1-t)\ga,D}
=
t(-L_{\be,D})
+
(1-t)(-L_{\ga,D}),
$$
 so by the same reasoning $t\be+(1-t)\ga\in
\h{\rm AA}_-(X,L,D)$.
\end{proof}

\noindent

\begin{remark}
{\rm 
One may wonder why we require $\aaxld$ to
be contained in the unit cube. 
Indeed, that is not an absolute must.
However, we are most interested in the 
``small angle limit" as $\be\ra 0\in\RR^r$.
Still, we require the coordinates to be 
positive (and not, say, limit to $0$ from
any orthant) since, geometrically, the 
$\be_i$ can sometimes be interpreted as the
cone angle associated to a certain class
of K\"ahler edge metrics.
One could
in principle allow the whole positive
orthant, still. But in this article
we restrict to the cube for practical
reasons.
}
\end{remark}

There are many interesting questions one can ask
about these convex bodies. For instance, how do
they transform under birational operations?
We now turn to describe a special, but important, 
situation where we will be able to use tools of 
convex optimization to say something about this question.

\section{Asymptotically log Fano/canonically polarized varieties}
\lb{ALFSec}

Perhaps the most important line bundles in algebraic geometry
are the canonical bundle of $X$, denoted $K_X$, and its dual, 
the anticanonical bundle, denoted $-K_X$.
These two bundles give rise to two extremely important 
classes of varieties: 

$\bullet\;$ Fano varieties are those for which  $-K_X$ is positive
\cite{Fano,IskProk},

$\bullet\;$ Canonically polarized (general type; minimal) 
varieties are those for which $K_X$ is positive \cite{Mats}
(big; nef).
Traditionally, algebraic geometers
have been trying to {\it classify} varieties
with positivity properties of $-K_X$ and to
{\it characterize} varieties with 
positivity properties of $K_X$. 
The subtle difference in terminology here stems
from the fact that positivity properties of $-K_X$
(think `positive Ricci curvature')
are rare and can sometimes be classified into a list
in any given dimension,
while positivity or bigness of $K_X$ is  much more
common, and hence a complete list is impossible, although
one can characterize such $X$ sometimes in terms of
certain traits. 
Be it as it may, the importance of these two classes
of varieties stems from the fact that, in some
very rough sense, the Minimal Model Program stipulates that
all projective varieties can be built from 
minimal/general type and Fano pieces. Put differently,
given a projective variety $K_X$ might not have a sign,
but one should be able to perform algberaic 
surgeries (referred to as {\it birational operations} or
{\it birational maps}) 
on it to eliminate the `bad regions'
of $X$ where $K_X$ is not well-behaved. Typically, these
birational maps will make $K_X$ more positive (in some sense
the common case, hence the terminology `general type'),
except in some rare cases when  $K_X$ is essentially negative to
begin with.

\subsection{Asymptotic logarithmic positivity 
associated to (anti)canonical divisors}

Thus, given the classical importance of positivity 
of $\pm K_X$, one may try to extend this to the logarithmic
setting. 

One may pose the following question:

\begin{question}
\label{Q8.1}
What are all triples $(X,D,\be)$ such that
$\be\in\h{\rm AA}_\pm(X,-K_X,D)$?
\end{question}

It turns out that the negative case of this question
is too vast to classify, and even the positive
case is out of reach unless we make some
further assumptions. 
We now try to at least
give some feeling for why this may be so, referring
to \cite[Question 8.1]{R14} for some further discussion.
At the end of the day, we will distill from
Question \ref{Q8.1} Problem \ref{aldpprob} which
we will then take up in the rest of this note.

First, without some restrictions on the parameter
$\be$ Question \ref{Q8.1} 
becomes too vast of a generalization which does
not seem to be extremely useful. For this reason%
\footnote{Another important reason is that that the
asymptotic logarithmic regime is closely related
to understanding differential-geometric limits, 
as $\be\ra0$ towards Calabi--Yau fibrations as conjectured
in \cite{CR,R14}},
we concentrate on the {\it asymptotic} logarithmic regime,
where $\be$ is required to be arbitrarily close to the origin.

\begin{definition}
\label{ALFDef}
{\rm \cite[Definition 1.1]{CR},\cite[Definition 8.13]{R14}}
$(X,D)$ is (strongly) asymptotically log Fano/canonically polarized 
if $(-K_X,D)$
is (strongly) asymptotically log positive/negative.
\end{definition}

\begin{remark} {\rm
\label{}
Definition \ref{ALFDef} is a special case, but, in fact, 
the main motivation for Definition \ref{ALPDef}.
The first, when $L=-K_X$, was
introduced by Cheltsov and the author \cite{CR}.
The second, when $L=K_X$, was introduced by the author \cite{R14}.
} \end{remark}

\begin{remark} {\rm
\label{}
When $(-K_X,D)$ is log positive one says $(X,D)$ is log Fano,
a definition due to Maeda \cite{Maeda}. By openess, 
log Fano is the most restrictive class, a subset of strongly 
asymptotically log Fano (ALF),
itself a subset of ALF.

} \end{remark}

\begin{remark} {\rm
\label{}
There is a beautiful differential geometric interpretation
of Definition \ref{ALFDef} in terms of Ricci curvature:
$(X,D)$ is asymptotically log Fano/general type if and only if $X$  admits a  K\"ahler metric with edge singularities of arbitrarily small angle $\be_i$ along each component $D_i$ of
the complex `hypersurface' $D$, and moreover the Ricci curvature of this \K
metric is positive/negative elsewhere. 
The only if part is an easy consequence 
of the definition \cite[Proposition 2.2]{DiCerbo}, the if part is
a generalization of the Calabi--Yau theorem conjectured by Tian
\cite{Tian1994}
and proved in \cite[Theorem 2]{JMR} when $D=D_1$, see also \cite{GP} for
a different approach in the general case (cf. \cite{MR}).
When $(X,D)$ is asymptotically log canonically polarized
the statement can even be improved to the existence of a \KE edge metric. 
We refer to \cite{R14} for exposition and a survey of these and other results.
} \end{remark}

Thus, the most
basic first step to understand Question \ref{Q8.1} 
becomes the following, posed in \cite{CR}.

\begin{problem}
\label{aldpprob}
Classify all ALF pairs $(X,D)$ with $\dim X=2$
and $D$ having simple normal crossings. 
\end{problem}

Asymptotically log Fano varieties in 
dimension 2 are often referred to as 
{\it asymptotically log del Pezzo surfaces}.
The simple normal crossings (snc) assumption is
a standard one in birational geometry and is
also the case that is of interest for the study
of \K edge metrics.

\subsection{Relation to the body of ample angles}
\lb{relationSec}

The problem of determining whether a given 
pair $(X,D=\sum_{i=1}^rD_i)$ is ALF amounts to determining whether the set 
$\h{\rm AA}_+(X,-K_X,D)$ satisfies
$$
0\in\overline{\h{\rm AA}_+(X,-K_X,D)}.
$$ 
Thus, the body of ample angles is a fundamental object in the
theory of asymptotically log Fano varieties.
This can also be rephrased in terms of intersection properties:
there exists $\eps_0>0$ such that 
$\h{\rm AA}_+(X,-K_X,D)\cap B(0,\eps)\not=\emptyset$
for all $\eps\in(0,\eps_0)$, where $B(0,\eps)$ is the
ball of radius $\eps$ centered at the origin in $\RR^r$.

If one replaces ``ALF" by ``strongly ALF" in Problem 
\ref{aldpprob} the problem has been solved
\cite[Theorems 2.1,3.1]{CR}. However, it turns out that in the
strong regime $\h{\rm AA}_+(X,-K_X,D)\subset \RR^4$
\cite[Corollary 1.3]{CR}.
In sum, the general case
is out of reach using only the methods of \cite{CR}:
in fact, in this note we will exhibit ALF pairs (which
are necessarily not strongly ALF) for which 
$\h{\rm AA}_+(X,-K_X,D)$ has arbitrary large 
dimension and outline a strategy for classifying
all ALF pairs.
We hope to complete this approach in ongoing joint work
with Martinez-Garcia \cite{MGR}.

Before describing our approach to Problem \ref{aldpprob}, let us pause to state 
an open  problem concerning these bodies
(for $X$ of any dimension).

\begin{problem}
\label{AAProb2}
How does $\h{\rm AA}_\pm(X,-K_X,D)$ behave under birational maps of $X$?
\end{problem}

\section{Convex optimization and classification in algebraic
geometry}
\lb{convexSec}

We finally get to the heart of this note where 
we show how birational operations on $X$ lead
to high-dimensional convex bodies.

To emphasize that we are in dimension $2$, from now on
we use the notation $(S,C)$ instead of $(X,D)$. Also,
since we are in the `Fano regime' we will drop the 
subscript `+'  and simply denote the body of ample
angles
$$
\aasc.
$$
We denote the twisted canonical class by (recall \eqref{Lbeta})
$$
K_{\beta,S,C}:=K_S+\sum_{i=1}^r(1-\be_i)C_i.
$$
The Nakai--Moishezon criterion stipulates 
that $\be\in \aasc$ if and only if 
\beq
\lb{NMCrit}
\h{$K_{\beta,S,C}^2>0$ and $K_{\beta,S,C}.Z<0$ for every 
irreducible algebraic curve $Z$ in $X$.}
\eeq 

The first is a single quadratic equation in $\be$
while the second is a possibly infinite system of linear
equations in $\be$. We will reduce both of these to
a finite system of linear equations. 

To that end
let us fix some ALF surface $(S,C)$, i.e., suppose
$
0\in\overline{\h{\rm AA}(S,C)}.
$ 
 We now ask:

\begin{question}
\label{blowupsQ}
What are all ALF pairs that can be obtained 
as blow-ups of $(S,C)$?
\end{question}

It turns out that there are infinitely-many such
pairs; the complete analysis is quite involved \cite{MGR}.
In this article we will exhibit a particular type
of (infinitely-many) such blow-ups that yields bodies of ample angles
of arbitrary dimension.

\subsection{Tail blow-ups}

A snc divisor $c$ in a surface is called a chain if
$c=c_1+\ldots+c_r$ with $c_1.c_2=\ldots=c_{r-1}.c_r=1$
and otherwise $c_i.c_j=0$ for all $i\not= j$. In our examples
each $c_i$ will be a smooth $\PP^1$. The singular points of $c$
are the $r-1$ intersection points; all other points on
$c$ are called its smooth points.

\begin{definition}
\label{TailDef}
We say that $(S,C)$ is a single tail blow-up of $(s,c)$
if $S$ is the blow-up of $s$ at a smooth point of $c_1\cup c_r$,
and $C=\pi^{-1}(c)$.
\end{definition}

Note that $C$ has $r+1$ components, the `new' component
being the exceptional curve $E=\pi^{-1}(p)$ where 
$p\in c_1\cup c_r$. If, without loss of generality, 
$p\in c_r$ then $E.\wt c_i=\delta_{ir}$,
so
$$
C=\wt c_1+\ldots+\wt c_r+E
$$
is still a chain.

As a very concrete example, we could take 
$S=\Fn$ and $C=Z_n+F$ (recall Example
\ref{Ex2}; when $n=0$ this is simply
$S=\PP^1\times\PP^1$
and $C=\{p\}\times\PP^1 + \PP^1\times\{q\}$, the snc
divisor (with intersection point $(p,q)$)).
There are two possible single tail blow-ups: 
blowing-up a smooth point 
either on $Z_n$ or on $F$.

\subsection{Towards a classification of nested ``tail" blow-ups}

In the notation of the previous paragraph, 
if $(S,C)$ is still ALF we could perform
another tail blow-up, blowing up a point on $c_1\cup E$,
and potentially repeat the process any number of times.
We formalize this in a definition.

\begin{definition}
\label{TailDef2}
We say that $(S,C)$ is an ALF tail blow-up of an ALF pair $(s,c)$
if $(S,C)$ is ALF and is obtained from $(s,c)$ as an iterated sequence of 
single tail blow-ups that result in ALF pairs in all intermediate steps.
\end{definition}

In other words, an ALF tail blow-up is 
a sequence of single tail blow-ups that
preserve asymptotic log positivity.

\begin{problem}
Classify all ALF tail blow-ups of ALF surfaces $(\Fn,c)$.
\end{problem}

The following result reduces the characterization of
ALF tail blow-ups to the feasibility of a certain linear program.

Define

\beq
\baeq
\h{LP}(S,C):=
\{
\be_x\in(0,1)^{r+x}\,:\, 
K_{\beta_x,S,C}.Z<0\quad 
&
\h{for every $Z\subset S$ such that
$\pi(Z)\subset s$ is a}
\cr
& \h{curve intersecting $c$ at finitely-many points}
\cr
&\h{and 
passing through the blow-up locus, and}
\cr
&
\mskip-125mu
K_{\beta_x,S,C}.C_i<0, \q i=1,\ldots,r+x.
\}
\eaeq
\eeq

\begin{theorem}
\label{TailThm0}
Let $(s,c)$ be an ALF pair.
An iterated sequence of $x$ 
single tail blow-ups $\pi:S\ra s$ of $(s,c)$ is an ALF tail blow-up
if only if (i) $x\le (K_s+c)^2$, and (ii) 
$0\in\overline{\h{\rm LP}(S,C)}$.
\end{theorem}

In fact, we will also show the following complementary result that
shows that (essentially) the only obstacle to completely characterizing tail
blow-ups are the (possibly) singular curves $Z$ passing through
the blow-up locus in the definition of $\h{\rm LP}(S,C)$.

Define
\beq
\baeq
\wt{\h{\rm LP}}(S,C):=
\{
\be_x\in(0,1)^{r+x}\,:\, 
K_{\beta_x,S,C}.C_i<0, \q i=1,\ldots,r+x
\}.
\eaeq
\eeq

\begin{theorem}
\label{TailThm1}
One always has $0\in\overline{\wt{\h{\rm LP}}(S,C)}$.
\end{theorem}

Before we embark on the proofs, a few remarks are in place.

\begin{remark} {\rm
\label{tailremark-0}
Observe that $(K_s+c)^2\ge 0$.
Indeed, since $(s,c)$ is ALF $-K_s-c$ is nef (as a limit of
ample divisors), so $(K_s+c)^2\ge0$.

} \end{remark}

\begin{remark} {\rm
\label{tailremark-1}
The proof will demonstrate that one can drop
``that result in ALF pairs in all intermediate steps"
from Definition \ref{TailDef2}, since it follows
from the fact that both $(s,c)$ and $(S,C)$ are ALF
(a sort of `interpolation' property).
} \end{remark}

\begin{remark} {\rm 
\label{tailremark-2}
We may assume that $c$ is a connected chain of $\PP^1$'s. Indeed,
when $(s,c)$ is ALF,
$c$ is either a cycle or a union of disjoint chains
\cite[Lemma 3.5]{CR} and each component is a $\PP^1$ 
\cite[Lemmas 3.2]{CR}. The former 
is irrelevant for us since there are no tails.
For the latter, we may assume that $c$ is connected (i.e.,
one chain) since
the only disconnected case, according to the classification
results \cite[Theorems 2.1,3.1]{CR}, is $(\Fn,c_1+c_2)$ with
$c_1=Z_n$ and $c_2\in|Z_n+nF|$ and then $(K_{\Fn}+c_1+c_2)^2=0$
so no tail blow-ups are allowed by Remark \ref{tailremark-1}. To see that,
let $c_1\in|aZ_n+bF|$ and $c_2\in|AZ_n+BF|$.
Since $c_1,c_2$ are effective, $b\ge na, B\ge nA$.
By assumption $c_1\cap c_2=\emptyset$ so
$0=c_1.c_2=-naA+aB+bA$, i.e., 
$
bA=a(nA-B)
$.
Since the right hand side is nonpositive and the
left hand side is nonnegative they must both be
zero, leading to
$
b=0, B=nA
$
($A=0$ is impossible since it would force $B=0$,
and $a=0$ is excluded by $b=0$).
Thus we see $c_1\in|aZ_n|, c_2\in|A(Z_n+nF)|$.
There are no smooth irreducible representatives of $|aZ_n|$ unless
$a=1$ and similarly for $|A(Z_n+nF)|$ unless $A=1$.

} \end{remark}

\begin{remark} {\rm 
\label{tailremark-3}
Theorems \ref{TailThm0} and \ref{TailThm1} are mainly given for illustrative reasons, i.e.,
to explicitly show how tools of convex programming can be used in this context.
As we show in a forthcoming extensive, but unfortunately long and 
tedious, classification work \cite{MGR} 
the case of tail blow-ups is in fact the ``worst" in terms of preserving 
asymptotic log positivity. We will give there a classification of asymptotically
log del Pezzo surfaces that completely avoids tail blow-ups since condition
(ii) in Theorem \ref{TailThm0}  is difficult to control, in general.
Thus, the present note and \cite{MGR} are somewhat complementary.
It is still an interesting open problem to classify all ALF tail blow-ups.
} \end{remark}

\subsection{The set-up}
\lb{setupsubs}

We start with an ALF pair $(s,c=c_1+\ldots+c_r)$
and perform $v+h$ single tail blow-ups 
of which 
\beq
\lb{droiteq}
\h{$h$ (`h\"ogra') tail blow-ups on the ``right tail" $c_r$}
\eeq
with associated blow-down map 
\beq
\lb{droit1eq}
\pi_H=\pi_1\circ\cdots\circ \pi_h
\eeq
and
exceptional curves 
\beq
\lb{droit2eq}
\exc(\pi_i)=H_i, \q i=1,\ldots,h,
\eeq
and of which
\beq
\lb{gaucheeq}
\h{$v$ (`v\"anster') tail blow-ups on the ``left tail" $c_1$}
\eeq
with blow-down map 
\beq
\lb{gauche1eq}
\pi_V=\pi_{h+1}\circ\cdots\circ \pi_{h+v}
\eeq
and exceptional curves 
\beq
\lb{gauche2eq}
\exc(\pi_{h+j})=V_j,\q i=1,\ldots,v,
\eeq
with new angles $\eta\in(0,1)^h$ and $\nu\in(0,1)^v$, respectively. 
Finally, we set
\beq
\lb{eta0nu0Eq}
\eta_0:=\be_r,\q
\nu_0:=\be_1.
\eeq

An induction argument shows:

\blemma
\lb{TailKXLemma}
With the notation \eqref{droiteq}--\eqref{eta0nu0Eq}, if $v,h>0$,
\begin{equation}
\begin{aligned}
\label{tailvhEq}
&-K_{(\be,\nu,\eta),S,(\pi_H\circ\pi_V)^{-1}(c)}
=\cr
&
-\pi_V^*\pi_H^*K_{\be,s,c}
-\sum_{i=1}^h(1-\eta_i+\eta_{i-1})\pi_V^*\pi_h^*\cdots\pi_{i+1}^*H_i
-\sum_{j=1}^v(1-\nu_j+\nu_{j-1})\pi_{h+v}^*\cdots\pi_{h+1+j}^*V_j
.
\end{aligned}
\end{equation}
If $v=0$,
\begin{equation}
\begin{aligned}
\label{tailhEq}
&-K_{(\be,\eta,\nu),S,\pi_H^{-1}(c)}
=
-\pi_H^*K_{\be,s,c}
-\sum_{i=1}^h(1-\eta_i+\eta_{i-1})\pi_h^*\cdots\pi_{i+1}^*H_i
.
\end{aligned}
\end{equation}
If $h=0$, 
\begin{equation}
\begin{aligned}
\label{tailvEq}
&-K_{(\be,\eta,\nu),S,\pi_V^{-1}(c)}
=
-\pi_V^*\pi_H^*K_{\be,s,c}
-\sum_{j=1}^v(1-\nu_j+\nu_{j-1})\pi_{v+h}^*\cdots\pi_{h+1+j}^*V_j
.
\end{aligned}
\end{equation}
\elemma

Before giving the proof, let us recall two elementary facts about blow-ups.
Let $\pi:S_2\ra S_1$ be the blow-up at a smooth point $p$ on
a surface $S_1$.
Then,
\beq
\lb{KXEq}
K_{S_2}=\pi^*K_{S_1}+E,
\eeq
where $E=\pi^{-1}(p)$ \cite[p. 187]{GH}, and
for every divisor $F\subset S_1$,
\beq
\lb{wtEq}
\wt F=\begin{cases}
\pi^*F, &\h{if $p\not\in F$,}\cr
\pi^*F-E, &\h{otherwise}.
\end{cases}
\eeq

\bpf
Using \eqref{KXEq}, if $v=0$,
\beq
\lb{KSheq}
K_{S}
=
\pi_{h}^*\Big(
\pi_{h-1}^*\big(
\cdots\big(
\pi_{1}^*\big(K_{s}+H_1)+H_2\big)
 +\ldots+ H_{h-2}\big)+H_{h-1}\Big)
+H_h.
\eeq
Similarly, if $h=0$,
\beq
\lb{KSveq}
K_{S}
=
\pi_{v}^*\Big(
\pi_{v-1}^*\big(
\cdots\big(
\pi_{1}^*\big(K_{s}+V_1)+V_2\big)
 +\ldots+ V_{v-2}\big)+V_{v-1}\Big)
+V_v.
\eeq
If $v,h>0$,
\beq
\lb{KSvheq}
\baeq
K_S
&=
\pi_{v+h}^*\Bigg(
\pi_{v+h-1}^*\bigg(
\cdots\Big(
\pi_{h+1}^*\big(
\pi_{h}^*\big(
\cdots\big(
\pi_{1}^*(K_{s}+H_1)+H_2\big)
\cr
&\qq\qq\qq  +\ldots+\big)
+H_h\big)+V_1\Big)+\ldots+V_{v-2}\bigg)+V_{v-1}\Bigg)
+V_v.
\eaeq
\eeq
Using \eqref{wtEq} and \eqref{eta0nu0Eq}, if $v=0$,
\beq
\lb{tailbndryhEq}
\baeq
\sum_{i=1}^{r+h}(1-\be_i)C_i
&=\sum_{i=1}^{r-1}(1-\be_i)\pi_H^*c_i
+
(1-\be_r)\pi_h^*\cdots\pi_2^*(\pi_1^*c_r-H_1)
\cr
&
\qq\q+(1-\eta_1)\pi_h^*\cdots\pi_3^*(\pi_2^*H_1-H_2)
+\ldots+(1-\eta_{h-1})(\pi_{H}^*H_{h-1}-H_h)
\cr
&
\qq\q+(1-\eta_h)H_h
\cr
&=
\sum_{i=1}^{r}(1-\be_i)\pi_H^*c_i
+
\sum_{i=1}^h(\eta_{i-1}-\eta_i)\pi_h^*\cdots\pi_{i+1}^*H_i,
\eaeq
\eeq
if $h=0$,
\beq
\lb{tailbndryvEq}
\baeq
\sum_{i=1}^{r+v}(1-\be_i)C_i
&=
(1-\be_1)\pi_{v+h}^*\cdots\pi_{h+2}^*(\pi_{h+1}^*c_1-V_1)
+
\sum_{i=2}^{r}(1-\be_i)\pi_V^*c_i
\cr
&
\qq+(1-\nu_1)\pi_{v+h}^*\cdots\pi_{h+2}^*(\pi_{h+1}^*V_1-V_2)
+\ldots+
(1-\nu_{v-1})(\pi_{v+h}^*V_{v-1}-V_v)
\cr
&
\qq+
(1-\nu_v)V_v
\cr
&=
\sum_{i=1}^{r}(1-\be_i)\pi_V^*c_i
+
\sum_{i=1}^v(\nu_{i-1}-\nu_i)\pi_v^*\cdots\pi_{i+1}^*V_i,
\eaeq
\eeq
and if $v,h>0$,
\beq
\baeq
\lb{tailbndryvhEq}
\sum_{i=1}^{r+v+h}(1-\be_i)C_i
&=
(1-\be_1)\pi_{v+h}^*\cdots\pi_{h+2}^*(\pi_{h+1}^*\pi_H^*c_1-V_1)
\cr
&
\qq+
\sum_{i=2}^{r-1}(1-\be_i)\pi_V^*\pi_H^*c_i
+
(1-\be_r)\pi_V^*\pi_h^*\cdots\pi_2^*(\pi_1^*c_r-H_1)
\cr
&
\qq+(1-\eta_1)\pi_V^*\pi_h^*\cdots\pi_3^*(\pi_2^*H_1-H_2)
+\ldots+(1-\eta_{h-1})\pi_V^*(\pi_{h}^*H_{h-1}-H_h)
\cr
&
\qq+(1-\eta_h)\pi_V^*H_h
\cr
&
\qq+(1-\nu_1)\pi_{v+h}^*\cdots\pi_{h+2}^*(\pi_{h+1}^*V_1-V_2)
+\ldots+
(1-\nu_{v-1})(\pi_{v+h}^*V_{v-1}-V_v)
\cr
&
\qq+
(1-\nu_v)V_v
\cr
=
\sum_{i=1}^{r}&(1-\be_i)\pi_V^*\pi_H^*c_i
+
\sum_{i=1}^h(\eta_{i-1}-\eta_i)\pi_V^*\pi_h^*\cdots\pi_{i+1}^*H_i
+
\sum_{i=1}^v(\nu_{i-1}-\nu_i)\pi_v^*\cdots\pi_{i+1}^*V_i.
\eaeq
\eeq
Thus, \eqref{KSvheq} and \eqref{tailbndryvhEq}  imply \eqref{tailvhEq},
\eqref{KSheq} and \eqref{tailbndryhEq}  imply \eqref{tailhEq},
and
\eqref{KSveq} and \eqref{tailbndryvEq}  imply \eqref{tailvEq}.
\epf

\bremark
\lb{dgpositiveremark}
In principle, as we will see below, 
the blow-ups on the left and on the right 
do not interact.
\eremark

\subsection{The easy direction and the sub-critical case}

We start with a simple observation. The easy direction of Theorem \ref{TailThm0} is contained in the next lemma:

\begin{lemma}
\label{necessitylemma}
Let $(s,c)$ be an ALF pair.
Let $(S,C)$ be obtained from $(s,c)$ 
via an iterated sequence of $x$ 
single tail blow-ups of $(s,c)$.
Then $(S,C)$ is not ALF  if $x> (K_s+c)^2$.

\end{lemma}

\bpf 
If $c$ does not contain a tail, there is nothing
to prove.
By Remark \ref{tailremark-1}, we may assume
that $c$ is a single chain.
Let $\pi:S\ra s $ denote the blow-up of a point on
a tail $c_r$ with exceptional curve $E=:C_{r+1}$. Then,
\begin{equation*}
\begin{aligned}
\label{}
-K_{(\be,\be_{r+1}),S,C+E}
&=
-\pi^*K_s-E-\sum_{i=1}^{r}(1-\be_i)\wt c_i
-(1-\be_{r+1})E
\cr
&=
-\pi^*K_s-E-\sum_{i=1}^{r-1}(1-\be_i)\pi^*c_i
-(1-\be_r)(\pi^*c_r-E)-(1-\be_{r+1})E
\cr
&=
-\pi^*K_{\be,s,c}
-(1+\be_r-\be_{r+1})E.
\end{aligned}
\end{equation*}
In particular, since $E^2=-1$, $K_{(0,0),S,C+E}^2=K_{0,s,c}^2-1$. 
An induction (or directly using Lemma \ref{TailKXLemma})
 thus shows that $(K_S+C)^2=(K_s+c)^2-x$,
which shows that $-K_S-C$ cannot be nef if $x>(K_s+c)^2$,
so $(S,C)$ cannot be ALF, by Remark \ref{tailremark-0}.
\epf

\subsection{Dealing with the quadratic constraint and the critical case}

Let 
$$
\be_x=(\be,\be_{r+1},\ldots,\be_{r+x})\in\RR^{r+x}.
$$
The proof of Lemma \ref{necessitylemma} also shows that 
$$
K_{\beta_x,S,C}^2=K_{\be,s,c}^2-x+f(\be_x),
$$
where $f:\RR^{r+x}\ra \RR$ is a quadratic polynomial
with no constant term and whose coefficients are integers
bounded by a constant depending only on $r+x$.
Thus, we also obtain some information regarding the converse
to Lemma \ref{necessitylemma}:

\bcor
\lb{subcriticalcor}
Let $(s,c)$ be an ALF pair.
Let $(S,C)$ be obtained from $(s,c)$ 
via an iterated sequence of $x$ 
single tail blow-ups of $(s,c)$.
Then $K_{\beta,S,C}^2>0$ for all sufficiently small (depending
only on $r,x$, hence only on $r,s,c$)
$\be_x\in\RR^{r+x}$ if $x< (K_s+c)^2$.
\ecor

This corollary is useful since it implies the quadratic
inequality in \eqref{NMCrit} can be completely ignored 
except, perhaps, in the borderline case $x=(K_s+c)^2$.

The next result treats precisely that borderline case:

\bprop
\lb{borderlineprop}
Let $(s,c)$ be an ALF pair.
Let $(S,C)$ be obtained from $(s,c)$ 
via an iterated sequence of $x:=(K_s+c)^2$ 
single tail blow-ups of $(s,c)$.
Then 
\beq
\lb{KbetaquadEq}
K_{\beta,S,C}^2=f(\be_x),
\eeq
where $f:\RR^{r+x}\ra \RR$ is a quadratic polynomial
with no constant term and whose coefficients are integers
bounded by a constant depending only on $r+x$, and moreover
it contains linear terms with positive coefficients and no
linear terms with negative coefficients.
In particular, $K_{\beta,S,C}^2>0$ for all sufficiently small (depending
only on $r,x$, hence only on $r,s,c$)
$\be_x\in(0,1)^{r+x}$.
\eprop

\begin{remark} {\rm
\label{tailremark-1}
The key for later will be \eqref{KbetaquadEq} rather than
the conclusion about $K_{\beta,S,C}^2>0$ for all sufficiently small
angles. In fact, the latter conclusion (at the end of
Proposition \ref{borderlineprop}) is not precise enough to
conclude that the quadratic inequality in \eqref{NMCrit}
 can be ignored as one needs that it holds {\it simultaneously}
with the intersection inequalities of  \eqref{NMCrit}. The exact
form of  \eqref{KbetaquadEq} implies that  \eqref{KbetaquadEq}
can be satisfied together with any {\it linear} constraints
on $\be_x$, which will be the key, and the reason that, ultimately,
the  quadratic inequality in \eqref{NMCrit} {\it can} be ignored.
} \end{remark}

\bpf
We use the notation of \S\ref{setupsubs}.
We wish to show that 
\begin{equation}
\begin{aligned}
\label{aequivtoEq}
K_{(\be,\delta,\ga),S,(\pi_H\circ\pi_V)^{-1}(c)}^2
>0, \q 
\h{for some small $(\be,\delta,\ga)\in(0,1)^{r+h+v}$}
\end{aligned}
\end{equation}
(recall $x=h+v=(K_s+c)^2$).
We compute,
\begin{equation}
\begin{aligned}
\label{}
K_{(\be,\delta,\ga),S,(\pi_H\circ\pi_V)^{-1}(c)}^2
&=
K_{\be,s,c}
-
\sum_{i=1}^h(1-\de_i+\de_{i-1})^2
-\sum_{j=1}^v(1-\ga_j+\ga_{j-1})^2
\cr
&=
(K_s+c)^2
-2\sum_{i=1}^r\be_i c_i.(K_s+c)
+\sum_{i=1}^r\be_i^2 c_i^2
\cr
&
\qq \qq 
-h+
2\sum_{i=1}^h\de_i
-
2\sum_{i=1}^h\de_{i-1}
-
v
+
2\sum_{j=1}^v\ga_j
-
2\sum_{j=1}^v\ga_{j-1}
\cr
&
\qq \qq 
-
\sum_{i=1}^h(\de_i-\de_{i-1})^2
-\sum_{j=1}^v(\ga_j-\ga_{j-1})^2
\cr
&=
-2\sum_{i=1}^r\be_i c_i.(K_s+c)
+
2\de_h-2\be_{r}
+
2\ga_v-2\be_1
-
O(\be^2,\de^2,\ga^2)
\cr
&=
2\be_1+2\be_r 
+
2\de_h-2\be_{r}
+
2\ga_v-2\be_1
-
O(\be^2,\de^2,\ga^2)
\cr
&=
2\de_h
+
2\ga_v
-
O(\be^2,\de^2,\ga^2)
,
\end{aligned}
\end{equation}
since, by Remark \ref{tailremark-2}, %
all
$c_i$ are smooth rational curves and $c$ is a single chain, 
so by adjunction
\begin{equation}
\begin{aligned}
\label{ciKscadjunEq}
c_i.(K_s+c)=
\begin{cases}
c_i.(K_s+c_i)+c_i.c_{i-1}+c_i.c_{i+1}=-2+1+1=0,
&
\h{if $i=2,\ldots,r-1$,}
\cr
c_r.(K_s+c_r)+c_r.c_{r-1}=-2+1=-1,
&
\h{if $i=r$,}
\cr
c_1.(K_s+c_1)+c_1.c_{2}=-2+1=-1,
&
\h{if $i=1$}.
\cr
\end{cases}
\end{aligned}
\end{equation}
This is clearly positive for $(\be,\de,\ga)=\eps(1,\ldots,1)$
for $\eps$ small enough. 
This proves the Proposition.
\epf

\bremark
As alluded to in the remark preceeding the proof,
one indeed can make $2\de_h
+
2\ga_v
-
O(\be^2,\de^2,\ga^2)
$ 
positive under any linear constraints on $\be,\de,\ga$ 
without imposing any new linear constraints
as the coefficients of the only non-zero linear terms are positive. 

\eremark

\subsection{Proof of Theorem \ref{TailThm0}}

First, suppose either (i) or (ii) does not hold.
If (i) fails then Lemma \ref{necessitylemma}
shows that $(S,C)$ is not ALF. If (ii) fails then
$(S,C)$ is not ALF by Definition \ref{ALFDef}.

Second, if both (i) and (ii) hold then  
Corollary \ref{subcriticalcor}, Proposition
\ref{borderlineprop}, and the Nakai--Moishezon 
criterion show that $(S,C)$ is ALF
if and only if $K_{\be_x,S,C}.Z<0$ for every
irreducible curve $Z\subset S$.
Naturally, 
we distinguish between three types of curves $Z$:

\smallskip
(a)  $\pi(Z)$ does not pass through the blow-up locus,

(b)  $\pi(Z)$ is contained in the blow-up locus,
 
(c)  $\pi(Z)$ is a curve passing through the blow-up locus.

\smallskip
Curves of type (a) can be ignored: Indeed, then $\pi(Z)$
is a curve in $s$ and $Z=\pi^*\pi(Z)$ (hence, does not intersect
any of the exceptional curves) so by Lemma \ref{TailKXLemma},
$$
K_{\beta,S,C}.Z=\pi^*K_{(\beta_1,\ldots,\be_r),s,c}.\pi^*\pi(Z)=
K_{(\beta_1,\ldots,\be_r),s,c}.\pi(Z).
$$ 
As $(s,c)$ is ALF, this intersection
number is negative.

Next, curves of type (c) are covered by condition (ii)
by the definition of $\h{\rm LP}(S,C)$. Finally,  since
curves of type (b) are, by definition of the tail blow-up, components
of the new boundary $C$, hence there are at most $x+2$ (i.e., finitely-many) of them,
certainly contained in the  finitely-many inequalities:
\beq
\lb{Kbeboundaryineq}
K_{\beta_x,S,C}.C_i<0, \q i=1,\ldots,r+x,
\eeq
which are once again covered by the definition of $\h{\rm LP}(S,C)$.
This concludes the proof of Theorem \ref{TailThm0}.

\subsection{Reduction of the linear intersection constraints}

In this subsection we  explain how to essentially further reduce the 
linear intersection constraints, i.e., we prove Theorem
\ref{TailThm1}. To that purpose, we show that curves of type 
(b) can be handled directly. This shows that the only potential
loss of asymptotic logarithmic positivity occurs from curves
of type (c) (observe that as in the previous subsection, curves
of type (a) can be ignored).

\begin{proof}[Proof of Theorem \ref{TailThm1}]
It suffices to check that
the system of $2r+2x$ inequalities 
\beq
\baeq
\lb{totalineq}
K_{\beta_x,S,C}.C_i
&<0, \q i=1,\ldots,r+x,
\cr
\be_i
&>0, \q i=1,\ldots,r+x,
\eaeq
\eeq
admit a solution along some ray emanating from the origin
in $\RR^{r+x}$. 

Let us first write these inequalities carefully
and by doing so eliminate some unnecessary ones.

Using Lemma \ref{TailKXLemma} we compute, starting with the tails, which turn out to pose
no constraints, to wit,
$$
\begin{aligned}
-K_{(\be,\delta,\ga),S,(\pi_H\circ\pi_V)^{-1}(c)}
.V_v
=
1-\ga_v+\ga_{v-1}>0,
\cr
-K_{(\be,\delta,\ga),S,(\pi_H\circ\pi_V)^{-1}(c)}
.\pi_V^*H_h
=
1-\de_h+\de_{h-1}>0.
\end{aligned}
$$
Next, we intersect with the other new boundary curves
(if $h,v>0$ there are $h+v-2$  such, if $h=0$ there are
$v-1$ such, if $v=0$ there are $h-1$ such),
\begin{equation}
\begin{aligned}
\label{dgminus2intersectEq}
-K_{(\be,\delta,\ga),S,(\pi_H\circ\pi_V)^{-1}(c)}
.\pi_{h+v}^*\cdots\pi_{h+j+1}^*(\pi_{h+j}^*V_{j-1}-V_j)
&=
(1-\ga_{j-1}+\ga_{j-2})
-
(1-\ga_j+\ga_{j-1})
\cr
&=
\ga_j-2\ga_{j-1}+\ga_{j-2},\q j=2,\ldots,v.
\cr
-K_{(\be,\delta,\ga),S,(\pi_H\circ\pi_V)^{-1}(c)}
.\pi_V^*\pi_h^*\cdots\pi_{i+1}^*(\pi_{i}^*H_{i-1}-H_i)
&=
(1-\de_{i-1}+\de_{i-2})
-
(1-\de_i+\de_{i-1})
\cr
&=
\de_i-2\de_{i-1}+\de_{i-2},\q i=2,\ldots,h.
\end{aligned}
\end{equation}
Finally, we intersect with the two `old tails'
(or only one if $\min\{h,v\}=0$),
and use \eqref{ciKscadjunEq},
\beq
\lb{final2ineq}
\begin{aligned}
-K_{(\be,\delta,\ga),S,(\pi_H\circ\pi_V)^{-1}(c)}
.
\pi_{h+v}^*\cdots\pi_{h+2}^*(\pi_{h+1}^*\pi_H^*c_1-V_1)
&=
-K_{\be,s,c}.c_1-
(1-\ga_1+\be_{1})
\cr
&=
1+\be_1c_1^2-(1-\ga_1+\be_{1})
\cr
&=\ga_1+(c_1^2-1)\be_1,
\cr
-K_{(\be,\delta,\ga),S,(\pi_H\circ\pi_V)^{-1}(c)}
.
\pi_V^*\pi_h^*\cdots\pi_2^*(\pi_1^*c_r-H_1)
&=
-K_{\be,s,c}.c_r-
(1-\de_1+\be_{r})
\cr
&=
1+\be_rc_r^2-(1-\de_1+\be_{r})
\cr
&=\de_1+(c_r^2-1)\be_r.
\end{aligned}
\eeq
\def\LP{\hbox{\rm LP}}
\def\Mat{\hbox{\rm Mat}}
Equations \eqref{dgminus2intersectEq}--\eqref{final2ineq}
are $h+v$ linear equations
that together with the $r+h+v$  constraints
$$
\be_x=(\be,\de,\ga)\in\RR_+^{r+h+v}
$$ 
can
be encoded by a $(r+h+v)$-by-$(r+2h+2v)$ matrix
inequality:
\begin{equation}
\begin{aligned}
\label{tailLPIneq}
(\be,\de,\ga)\LP(S,(\pi_H\circ\pi_V)^{-1}(c))>0,
\end{aligned}
\end{equation}
where the inequality symbol means that each component of the vector is positive
(typical notation in linear optimization, see, e.g., \cite{Dantzig}) with
$$
\LP(S,(\pi_H\circ\pi_V)^{-1}(c)):=
\begin{cases}
\begin{pmatrix}
v_r&v_1&T&I_{r+h+v}
\end{pmatrix} & \h{ if $h,v>0$},\cr
\cr
\begin{pmatrix}
v_r&T&I_{r+h}
\end{pmatrix} & \h{ if $h>0,v=0$},\cr
\cr
\begin{pmatrix}
v_1&T&I_{r+v}
\end{pmatrix} & \h{ if $h=0,v>0$},\cr
\end{cases}
$$
where 
$$ 
\begin{aligned}
v_r&=
(\overbrace{0,\ldots,0}^{r-1},c_r^2-1,1,\overbrace{0,\ldots,0}^{h-1},\overbrace{0,\ldots,0}^{v})^T
\in \ZZ^{h+v+r},
\cr
v_1&=
(c_1^2-1,\overbrace{0,\ldots,0}^{r-1},\overbrace{0,\ldots,0}^{h},1,\overbrace{0,\ldots,0}^{v-1})^T
\in \ZZ^{h+v+r},
\cr
T&=
\begin{cases}
\begin{pmatrix}
T_r&T_1\cr
T_h&0_{h,v-1}\cr
0_{v,h-1}&T_v\cr
\end{pmatrix}\in\Mat_{r+h+v,h+v-2},
\h{ if $h,v>0$}
\cr
\begin{pmatrix}
T_r\cr
T_h\cr
\end{pmatrix}\in\Mat_{r+h,h-1},
\h{ if $h>0,v=0$}
\cr
\begin{pmatrix}
T_1\cr
T_v\cr
\end{pmatrix}\in\Mat_{r+v,v-1},
\h{ if $h=0,v>0$}
\end{cases}
\end{aligned}
$$
with 
$$ 
\begin{aligned}
T_r&=
\begin{pmatrix}
0&0&\ldots&0\cr
\vdots&\vdots&\ldots&\vdots\cr
0&0&\ldots&0\cr
1&0&\ldots&0\cr
\end{pmatrix}\in\Mat_{r,h-1},
\q\qq
T_1=
\begin{pmatrix}
1&0&\ldots&0\cr
0&0&\ldots&0\cr
\vdots&\vdots&\ldots&\vdots\cr
0&0&\ldots&0\cr
\end{pmatrix}\in\Mat_{r,v-1},
\cr
T_h&=
\begin{pmatrix}
-2&1&\ldots&0\cr
1&-2&\ldots&0\cr
0&1&\ldots&0\cr
\vdots&\vdots&\ldots&\vdots\cr
0&\ldots&0&1\cr
0&0&\ldots&-2\cr
0&\ldots&0&1\cr
\end{pmatrix}\in\Mat_{h,h-1},
\q
T_v=
\begin{pmatrix}
-2&1&\ldots&0\cr
1&-2&\ldots&0\cr
0&1&\ldots&0\cr
\vdots&\vdots&\ldots&\vdots\cr
0&\ldots&0&1\cr
0&0&\ldots&-2\cr
0&\ldots&0&1\cr
\end{pmatrix}\in\Mat_{v,v-1},
\end{aligned}
$$
(here, we use the convention that $T_r$ and $T_h$ are the empty matrix
if $h<2$ and similarly for $T_1$ and $T_v$ if $v<2$).

By Gordan's Theorem \cite[p. 136]{Dantzig}, the 
inequalities \eqref{tailLPIneq} hold if and only if
the only solution $y\in\RR_+^{r+2h+2v}$ to 
$$
\LP(S,(\pi_H\circ\pi_V)^{-1}(c))y=0
$$
is $y=0\in\RR_+^{r+2h+2v}$. 
We treat first the (easy) cases 
$$
(h,v)\in\{(1,0),(0,1),(2,0),(0,2),(1,1),(2,1),(1,2)\}
$$
separately. 

The case $(1,0)$ imposes only the inequality
$\de_1+(c_r^2-1)\be_r>0$ which is feasible.
Similarly, the case $(0,1)$ imposes 
only $\ga_1+(c_1^2-1)\be_1>0$.
The case $(1,1)$ imposes both of these
inequalities, but they are independent, hence feasible.

The case $(2,0)$ imposes the inequalities
\beq
\lb{twozeroineq}
\de_1+(c_r^2-1)\be_r>0, 
\q
\de_2-2\de_1+\be_r>0,
\eeq
which are equivalent via a Fourier--Motzkin elimination  \cite[\S4.4]{Dantzig} to
$
\de_2+\be_r>
2(1-c_r^2)\be_r,
$
i.e., $\de_2>(1-2c_r^2)\be_r$, which is feasible. 
The case $(0,2)$ is handled similarly.
The case $(2,2)$ is feasible
for the same reasons: both sets of inequalities
are feasible and independent.
The case $(2,1)$ (and similarly $(1,2)$)
also follows since it imposes the inequalities
\eqref{twozeroineq} in addition to the independent
inequality $\ga_1+(c_1^2-1)\be_1>0$, thus these
are feasible. This idea of independence will also be useful in the general
case below.

Let us turn to the general case, i.e., suppose $h,v\ge2$.
First, the $r+h$-th row of $\LP(S,(\pi_H\circ\pi_V)^{-1}(c))$ is
$$
(\overbrace{0,\ldots,0}^{h},1,\overbrace{0,\ldots,0}^{v-1+r+h-1},
1,\overbrace{0,\ldots,0}^{v})
.
$$
This implies $y_{h+1}=y_{r+2h+v-1}=0$.
If $h=2$ this implies $y_1=y_{r+2h+v-2}=0$; if $h>3$ this implies
$y_{h}=y_{r+2h+2v-2}=0$ (the $-2$ in the $(h+1)$-th spot in that 
row is taken care of by the fact $y_{h+1}=0$ from the previous step), 
and inductively
we obtain 
$y_{h+1-i}=y_{r+2h+2v-i}=0, \; i=1,\ldots, h-2$,
and finally $y_1=y_{r+h+2v+1}=0$.
 Altogether, we have shown $2h$ of the $y_i$'s are zero.

Second,  the $r+h+v$-th (last) row is
$$
(\overbrace{0,\ldots,0}^{h+v-1},1,\overbrace{0,\ldots,0}^{r+h+v-1},1)
.
$$
This implies   $y_{h+v}=y_{r+2h+2v}=0$.
If $v>2$ this implies
$y_{h+v-1}=y_{r+2h+2v-1}=0$, and inductively
we obtain 
$y_{h+v-i}=y_{r+2h+2v-i}=0, \; i=1,\ldots, v-2$,
and finally $y_2=y_{r+2h+v+1}=0$.
In this step we have shown $2v$ of the $y_i$'s are zero.

So far we have shown $2h+2v$ of the  $y_i$'s are zero using
the last $2h+2v$ rows. 

Finally, we consider the first $r$ rows. There 
are   two special rows with possibly positive
coefficients $c_r^2-1$ and
$c_1^2-1$, however the corresponding $y_1$ and $y_2$ are zero,
so as we have the full rank and identity matrix $I$ (with nonnegative
coefficients) in 
$\LP(S,(\pi_H\circ\pi_V)^{-1}(c))$
it follows that the remaining $r$ variables $y_i$ are zero, concluding 
the proof of Theorem \ref{TailThm1}.
\end{proof}

\def\bi#1{\bibitem{#1}}
	
	\begin{spacing}{0.1}

	\end{spacing}
	
	\bigskip
	
	{\sc University of Maryland}
	
	{\tt yanir@umd.edu}

\end{document}